\documentclass[10pt]{article}
\usepackage[a4paper,centering,hscale=0.75,vscale=0.75]{geometry}

\usepackage{lmodern}
\usepackage[T1]{fontenc}

\usepackage{mathtools}      
\usepackage{amsthm,amssymb} 
\usepackage{mathrsfs}
\usepackage{bbm}

\newtheorem{prop}{Proposition}[section]
\newtheorem{thm}[prop]{Theorem}
\newtheorem{lemma}[prop]{Lemma}
\newtheorem{coroll}[prop]{Corollary}

\theoremstyle{remark}
\newtheorem{rmk}[prop]{Remark}

\newcommand{\E}{\mathop{{}\mathbb{E}}\nolimits}

\newcommand{\cF}{\mathscr{F}}

\renewcommand{\P}{\mathbb{P}}

\newcommand{\erre}{\mathbb{R}}
\newcommand{\cX}{\mathscr{X}}

\renewcommand{\geq}{\geqslant}
\renewcommand{\leq}{\leqslant}

\newcommand{\indp}[1]{\mathbbm{1}_{\{#1\}}}

\numberwithin{equation}{section}

\DeclarePairedDelimiter\abs{\lvert}{\rvert}
\DeclarePairedDelimiter\norm{\lVert}{\rVert}

\DeclarePairedDelimiterX\ip[2]{\langle}{\rangle}{#1,#2}

\mathtoolsset{centercolon}

\makeatletter
\renewenvironment{enumerate}{%
  \ifnum \@enumdepth >3 \@toodeep\else
      \advance\@enumdepth \@ne
      \edef\@enumctr{enum\romannumeral\the\@enumdepth}\list
      {\csname label\@enumctr\endcsname}{\usecounter
        {\@enumctr}\def\makelabel##1{\hss\llap{\upshape##1}}}\fi
}{%
  \endlist
}

\renewenvironment{itemize}{%
  \ifnum\@itemdepth>3 \@toodeep
  \else \advance\@itemdepth\@ne
    \edef\@itemitem{labelitem\romannumeral\the\@itemdepth}%
    \list{\csname\@itemitem\endcsname}%
      {\def\makelabel##1{\hss\llap{\upshape##1}}}%
  \fi
}{%
  \endlist
}

\makeatother

\begin{document}

\title{On some semi-parametric estimates for \\ European option prices}

\author{Carlo Marinelli\thanks{Department of Mathematics, University
    College London, Gower Street, London WC1E 6BT, UK.}}

\date{\normalsize June 19, 2023}

\maketitle

\begin{abstract}
  We show that an estimate by de la Pe\~na, Ibragimov and Jordan for
  \(\E(X-c)^+\), with \(c\) a constant and \(X\) a random variable of
  which the mean, the variance, and \(\P(X \leq c)\) are known,
  implies an estimate by Scarf on the infimum of \(\E(X \wedge c)\)
  over the set of positive random variables \(X\) with fixed mean and
  variance. This also shows, as a consequence, that the former
  estimate implies an estimate by Lo on European option prices.
\end{abstract}


\section{Introduction}
A remarkable result by Scarf \cite{Scarf} provides an explicit
solution to the problem of minimizing \(\E(X \wedge c)\), with \(c\) a
positive constant, over the set of all positive random variables \(X\)
with given mean and variance (see Theorem \ref{thm:Scarf} below). The
infimum is shown to have two different expressions depending on
whether the parameter \(c\) is above or below a certain threshold and
to be attained by a random variable taking only two values. About
thirty years later after the publication of \cite{Scarf}, Lo
\cite{Lo87} noticed that Scarf's result immediately implies an upper
bound for \(\E(X-c)^+\), with \(X\) and \(c\) as before. This has an
obvious financial interpretation as an upper bound for the price at
time zero of a European call option with strike \(c\) on an asset with
value at maturity equal to \(X\) (in discounted terms, assuming that
expectation is meant with respect to a pricing measure). More
recently, de la Pe\~na, Ibragimov and Jordan \cite{dlP:opt} obtained,
among other things, a sharp upper bound for \(\E(X-c)^+\) over the set
of random variables \(X\) for which mean and variance as well as the
probability \(\P(X \leq c)\) are known (see Theorem \ref{thm:dlP}
below).

Our goal is to prove that the estimate by de la Pe\~na, Ibragimov and
Jordan is stronger than Scarf's estimate, in the sense that the former
implies the latter. This may appear somewhat counterintuitive, as the
former estimate requires the extra input \(\P(X \leq c)\), while the
latter has an extra positivity constraint.

The proof by Scarf, while relatively elementary, is quite ingenious. A
different proof, covering also substantial generalizations, has been
obtained by Bertsimas and Popescu \cite{BePo} using duality methods in
semi-definite optimization. The arguments used in \cite{dlP:opt} are
instead based on classical probabilistic inequalities and a ``toy''
version of decoupling.
The proof of Scarf's estimate given here is entirely elementary and
self-contained. Starting from an alternative proof of the relevant
estimates in \cite{dlP:opt}, another proof of Scarf's result is
obtained that is certainly not as deft as the original, but that would
hopefully seem more natural to anyone who, like the author, would
hardly ever come up with the ingenious idea used in \cite{Scarf}.  Our
proof is based, roughly speaking, on a representation of the set of
random variables with given mean and variance as union of subsets of
equivalent random variables, where two random variables \(X_1\) and
\(X_2\) are equivalent if \(\P(X_1 \leq c) = \P(X_2 \leq c)\). This
allows to establish a link between the two inequalities and to reduce
the problem of proving a version of Scarf's result without the
positivity constraint to the minimization of a function of one real
variable. Finally, the positivity constraint is taken into account,
thus establishing the full version of Scarf's result. Moreover, the
fact that optimizers exist and are given by two-point distributed
random variables appears in a natural way and plays an important role
in the proof.

The result proved in this article may also clarify, or at least
complement, several qualitative remarks made in \cite{dlP:opt} about
the relation between the two above-mentioned inequalities. For
instance, the authors note that their inequality is simpler than Lo's
inequality in the sense that the right-hand side does not depend on
the value of \(c\). Here we show that this is only due to the
positivity constraint in \cite{Scarf}, with very explicit calculations
showing how the threshold value for \(c\) arises. Moreover, the
sharpness of their inequality is proved in a much more natural way,
i.e. showing that a two-point distributed random variable always
attains the bound.

Apart from the application to bounds for option prices, estimates of
(functions of) \(X \wedge c\), sometimes called the Winsorization of
\(X\), are important in several areas of applied probability. For
results in this direction, as well as for an informative discussion
with references to the literature, we refer to \cite{Pine:tr}.


\section{Preliminaries}
Let \((\Omega,\cF,\P)\) be a probability space, on which all random
elements will be defined. We shall write, for simplicity, \(L^2\) to
denote \(L^2(\Omega,\cF,\P)\), and \({\norm{\cdot}}_2\) for its norm.
For any \(m \in \erre\) and \(\sigma \in \erre_+\), the sphere of
\(L^2\) of radius \(\sigma\) centered in \(m\) will be denoted by
\(\cX_{m,\sigma}\), and just by \(\cX\) if \(m=0\) and \(\sigma=1\).
In other words, \(\cX_{m,\sigma}\) stands for the set of random
variables \(X\) such that \(\E X=m\) and
\(\operatorname{Var}(X) = \E(X-m)^2 = \sigma^2\).  It is clear that
\(\cX_{m,\sigma} = m + \sigma \cX\).
The intersection of \(\cX_{m,\sigma}\) with the cone of random
variables bounded below by \(\alpha \in \erre\) will be denoted by
\(\cX_{m,\sigma}^\alpha\).  It is easily verified that, for any
\(m \in \erre\) and \(\sigma \in \erre_+\),
\begin{equation}
  \label{eq:aff}
  m + \sigma \cX^\alpha = \cX_{m,\sigma}^{m+\sigma\alpha}.  
\end{equation}

Recall that a random variable is said to be two-point distributed if
it takes only two (distinct) values. The set of two-point distributed
random variables belonging to \(\cX\) can be parametrized by the open
interval \(\mathopen]0,1\mathclose[\): let \(X\) take the values
\(x,y \in \erre\), \(x < y\), and set
\[
  p := \P(X=x), \qquad 1-p = \P(X=y).
\]
Then \(X \in \cX\) if and only if
\[
px + (1-p)y = 0, \qquad px^2 + (1-p)y^2 = 1,
\]
which imply
\begin{equation}
  \label{eq:xy}
  x = - \biggl( \frac{1-p}{p} \biggr)^{1/2}, \qquad
  y = \biggl( \frac{p}{1-p} \biggr)^{1/2}.
\end{equation}
Note that \(p=0\) and \(p=1\) are not allowed, hence
\(p \in \mathopen]0,1\mathclose[\). This is also obvious a priori, as
there is no degenerate random variable with mean zero and variance
one.
The following simple observations, the proof of which is an immediate
consequences of \eqref{eq:xy} and elementary algebra, will be useful.
\begin{lemma}
  \label{lm:2p}
  Let \(X \in \cX\) be the \(\{x,y\}\)-valued random variable
  identified by the parameter \(p \in \mathopen]0,1\mathclose[\), and
  \(c \in \erre\).
  \begin{itemize}
  \item[(i)] If \(c \geq 0\), then
    \[
      x < c \leq y \quad \text{if and only if} \quad
      p \geq \frac{c^2}{1+c^2}.
    \]
  \item[(ii)] If \(c < 0\), then
    \[
      x \leq c < y \quad \text{if and only if} \quad
      p \leq \frac{1}{1+c^2}.
    \]
  \end{itemize}
\end{lemma}

We shall also need the following elementary lattice identities.
\begin{lemma}
  \label{lm:latt}
  Let \(a,b,c \in \erre\). The following holds:
    \begin{itemize}
  \item[(i)] \(a + (b \wedge c) = (a+b) \wedge (a+c)\);
  \item[(ii)] if \(a \geq 0\), then
    \(a (b \wedge c) = (ab) \wedge (ac)\);
  \item[(iii)] \((a-b)^+ = a - (a \wedge b)\)
  \end{itemize}
\end{lemma}
\begin{proof}
  The identities in (i) and (ii) are clear. The identity in (iii) can
  be verified ``case by case'', but it can also be deduced from the
  identity
  \[
    (a-b)^+ = (a-b) + (a-b)^-,
  \]
  where, using (i),
  \[
    (a-b)^- = - \bigl( (a-b) \wedge 0 \bigr) = b - (a \wedge b),
  \]
  from which the claim follows immediately.
\end{proof}


\section{de la Pe\~na-Ibragimov-Jordan bound}
The following sharp estimates are proved in
\cite{dlP:opt}.
\begin{thm}[de~la~Pe\~na, Ibragimov, and Jordan]
  \label{thm:dlP}
  Let \(X \in \cX_{m,\sigma}\) and \(c \in \erre\). Setting
  \(p_0 := \P(X>c)\), one has
  \begin{equation}
    \label{eq:dlP}
    (m-c)p_0 \leq \E(X-c)^+ \leq (m-c)p_0 + \sigma(p_0-p_0^2)^{1/2}.
  \end{equation}
\end{thm}
The proof of \eqref{eq:dlP} in \cite{dlP:opt} is very elegant, the
main idea being the introduction of an independent copy of the random
variable $X$.  Here we give an alternative, entirely elementary proof.
\begin{proof}
  Let us start with the lower bound. One can assume, without loss of
  generality, that $m \geq c$, otherwise there is nothing to prove.
  Since
  \[
    \E(X-c)^+ = \E(X-c)\indp{X>c}
    = \E X\indp{X>c} - c\P(X>c),
  \]
  it suffices to show that \(\E X \indp{X>c} \geq m \P(X>c)\).
  To this purpose, note that
  \[
    \E X \indp{X>c} = \E X - \E X \indp{X \leq c},
  \]
  where, thanks to the assumption \(m \geq c\),
  \[
    \E X \indp{X \leq c} \leq \E c \indp{X \leq c} = c \P(X \leq c)
    \leq m \P(X \leq c),
  \]
  hence
  \[
    \E X \indp{X>c} \geq m - m\P(X \leq c) = m \P(X>c).
  \]
  To prove the upper bound, note that
  \[
    \E(X-c)\indp{X>c} - (m-c)p_0
    = \E\bigl( X\indp{X>c} - m\indp{X>c} \bigr),
  \]
  hence, adding and subtracting \(\E Xp_0 = mp_0\) on the right-hand side,
  Cauchy-Schwarz's inequality yields
  \begin{align*}
    \E(X-c)^+ - (m-c)p_0
    &= \E(X-m)(\indp{X>c}-p_0)\\
    &\leq \sigma \norm[\big]{\indp{X>c}-p_0}_2\\
    &= \sigma (p_0-p_0^2)^{1/2},
  \end{align*}
  thus completing the proof.
\end{proof}

\begin{rmk}
  Even if \(m < c\), the inequality \(\E X \indp{X>c} \geq m \P(X>c)\)
  is still true. In fact,
  \[
    m \P(X>c) \leq c \P(X>c) = \E c \indp{X>c} \leq \E X \indp{X>c}.
  \]
\end{rmk}

Theorem \ref{thm:dlP} implies useful one-sided Chebyshev-like bounds.
\begin{coroll}
  \label{cor:Xpc}
  Let \(X \in \cX\) and \(c \in \erre\). The following holds:
  \begin{itemize}
  \item[(i)] if \(c \geq 0\), then
    \(\displaystyle \P(X \leq c) \geq \P(X < c) \geq \frac{c^2}{1+c^2}\);
  \item[(ii)] if \(c<0\), then
    \(\displaystyle \P(X \leq c) \leq \frac{1}{1+c^2}\).
  \end{itemize}
\end{coroll}
\begin{proof}
  The proof of Theorem \ref{thm:dlP} remains valid with
  \(p_1 := \P(X \geq c)\) in place of \(p_0\), hence, as the
  right-hand side of \eqref{eq:dlP} must be positive,
  \[
    \sqrt{p_1(1-p_1)} \geq cp_1.
  \]
  If \(c \geq 0\), squaring both sides yields a linear inequality that
  is satisfied if and only if \(p_1 \leq 1/(1+c^2)\). This proves (i).
  \smallskip\par\noindent
  (ii) If \(c<0\),
  \[
    \P(X \leq c) = \P(-X \geq -c) = 1 - \P(-X < -c).
  \]
  Since \(-X \in \cX\) and \(-c>0\), (i) implies that
  \[
    \P(-X < -c) \geq \frac{c^2}{1+c^2},
  \]
  hence
  \[
    \P(X \leq c) \leq 1 - \frac{c^2}{1+c^2} = \frac{1}{1+c^2}.
    \qedhere
  \]
\end{proof}

\begin{rmk}
  Let \(X \in \cX\) and \(c \in \erre_+\). By a reasoning entirely
  analogous to the proof of (ii) in the previous corollary, both
  \(\P(X > c)\) and \(\P(X < -c)\) are bounded above by \(1/(1+c^2)\),
  hence \(\P(\abs{X}>c) \leq 2/(1+c^2)\), which is sharper than
  Chebyshev's inequality \(\P(\abs{X}>c) \leq 1/c^2\) if \(c<1\).
\end{rmk}


\section{Scarf-Lo bound}
The following estimate is obtained in \cite{Scarf}.
\begin{thm}[Scarf]
  \label{thm:Scarf}
  Let \(c,m,\sigma\) be strictly positive real numbers. The infimum of
  the function \(X \mapsto \E(X \wedge c)\) on the set
  \(\cX_{m,\sigma}^0\) is attained, i.e. it is a minimum, and is given
  by
  \[
    \min_{X \in \cX_{m,\sigma}^0} \E(X \wedge c) =
    \begin{cases}
      \displaystyle \frac{m^2}{m^2+\sigma^2}c,
      & \text{if } \displaystyle c \leq \frac{m^2+\sigma^2}{2m},\\[10pt]
      \displaystyle \frac{c+m}{2} - \frac12 \bigl( (c-m)^2 + \sigma^2
      \bigr)^{1/2}, & \text{if } \displaystyle c \geq
      \frac{m^2+\sigma^2}{2m}.
    \end{cases}
  \]
\end{thm}
Note that the constraint \(X \geq 0\) is dictated by the structure of
the practical inventory problem considered by Scarf. It is not needed,
however, to avoid the infimum to be minus infinity. In fact,
\[
  \abs{\E(X \wedge c)} \leq \E\abs{X \wedge c} \leq \E\abs{X} +
  \abs{c} \leq \bigl(\E X^2 \bigr)^{1/2} + \abs{c},
\]
where
\[
  \bigl(\E X^2 \bigr)^{1/2} = \norm{X}_2 = \norm{X-m+m}_2 \leq \sigma
  + \abs{m},
\]
which implies
\[
  \inf_{X \in \cX_{m,\sigma}} \E(X \wedge c) \geq -(\sigma + \abs{m} +
  \abs{c}).
\]

As observed in \cite{Lo87}, identity (iii) of Lemma \ref{lm:latt}
immediately yields the
\begin{coroll}[Lo]
  Let \(c,m,\sigma\) be strictly positive real numbers. The supremum
  of the function \(X \mapsto \E(X - c)^+\) on the set
  \(\cX_{m,\sigma}^0\) is attained, i.e. it is a maximum, and is given
  by
  \[
    \max_{X \in \cX_{m,\sigma}^0} \E(X - c)^+ =
    \begin{cases}
      \displaystyle m - \frac{m^2}{m^2+\sigma^2}c,
      & \text{if } \displaystyle c \leq \frac{m^2+\sigma^2}{2m},\\[10pt]
      \displaystyle \frac{m-c}{2} + \frac12 \bigl( (m-c)^2 + \sigma^2
      \bigr)^{1/2}, & \text{if } \displaystyle c \geq
      \frac{m^2+\sigma^2}{2m}.
    \end{cases}
  \]
\end{coroll}

The remainder of this section is dedicated to showing that Theorem
\ref{thm:Scarf}, hence also its corollary, are a consequence of
Theorem \ref{thm:dlP}. We shall argue by a sequence of elementary
lemmas and propositions. The first is a reduction step that, in spite
of its simplicity, considerably reduces the burden of symbolic
calculations. Throughout the section we assume that
\(c,m,\sigma \in \erre\), with \(\sigma>0\). Further constraints (that
do not imply any loss of generality) will be introduced as needed.
\begin{lemma}
  \label{lm:redu}
  Let
  \[
    \widetilde{c} := \frac{c-m}{\sigma}.
  \]
  Then
  \[
    \inf_{X \in \cX_{m,\sigma}^0} \E(X \wedge c) =
    m + \sigma \inf_{X \in \cX^{-m/\sigma}} \E(X \wedge \widetilde{c}).
  \]
\end{lemma}
\begin{proof}
  Since \(\cX^0_{m,\sigma} = m + \sigma \cX^{-m/\sigma}\) by
  \eqref{eq:aff}, one has
  \[
    \inf_{X \in \cX_{m,\sigma}^0} \E(X \wedge c) =
    \inf_{Y \in \cX^{-m/\sigma}} \E( (m+\sigma Y) \wedge c),
  \]
  where, by Lemma \ref{lm:latt},
  \[
    \E( (m+\sigma Y) \wedge c) = m + \sigma \E\biggl(
    Y \wedge \frac{c-m}{\sigma} \biggr),
  \]
  which immediately yields the conclusion.
\end{proof}

The lemma implies that it suffices to study the problem of minimizing
the function \(X \mapsto \E(X \wedge c)\) over the set \(\cX^{-m}\).
We shall first study the minimization problem without the
lower-boundedness constraint, i.e. on \(\cX\) rather than on
\(\cX^{-m}\). 
We shall need some more notation: the subset of \(\cX_{m,\sigma}\)
such that \(\P(X \leq c) = p\) is denoted by
\(\cX_{m,\sigma}(p;c)\). Note that, in view of Corollary
\ref{cor:Xpc}, these sets are non-empty only for certain combinations
of the parameters \(p\) and \(c\).
Let \(L_c \colon \mathopen]0,1\mathclose[ \mapsto \erre\) be the
function defined by
\[
  L_c\colon p \longmapsto - (p-p^2)^{1/2} + c(1-p).
\]
\begin{lemma}
  \label{lm:iL}
  One has
  \[
    \inf_{X \in \cX(c;p)} \E(X \wedge c) \geq L_c(p)
  \]
\end{lemma}
\begin{proof}
  Lemma \ref{lm:latt}(iii) implies
  \[
    \E (X-c)^+ = -\E(X \wedge c)
  \]
  for any \(X\) with mean zero, hence, by Theorem \ref{thm:dlP},
  \[
    \inf_{X \in \cX(c;p)} \E(X \wedge c)
    = - \sup_{X \in \cX(c;p)} \E (X-c)^+
    \geq - (p-p^2)^{1/2} + c(1-p).
    \qedhere
  \]
\end{proof}

We are now going to show that the infimum in Lemma \ref{lm:iL} is
achieved, and that the minimizer is a two-point distributed random
variable. We shall only consider, without loss of generality, those
values of \(p\) such that \(\cX(p;c)\) is non-empty, that is, by
Corollary \ref{cor:Xpc}, setting
\[
\Pi_c :=
\begin{cases}
  \displaystyle \Bigl] 0 , \frac{1}{1+c^2} \Bigr],
  &\quad \text{if } c<0,\\[6pt]
  \displaystyle \Bigl[ \frac{c^2}{1+c^2} , 1 \Bigr[,
  &\quad \text{if } c \geq 0,
\end{cases}
\]
to \(p \in \Pi_c\).
\begin{lemma}
  \label{lm:45}
  Let \(p \in \Pi_c\) and \(X_0 \in \cX\) be the two-point
  distributed random variable with parameter \(p\). Then
  \[
    \E(X_0 \wedge c) = L_c(p),
  \]
  hence
  \[
    \inf_{X \in \cX(c;p)} \E(X \wedge c)
    = \min_{X \in \cX(c;p)} \E(X \wedge c)
    = \E(X_0 \wedge c) = L_c(p).
  \]
\end{lemma}
\begin{proof}
  Since \(p \in \Pi_c\), Lemma \ref{lm:2p} implies that the random
  variable \(X_0\), taking the values \(x\) and \(y\) as defined in
  \eqref{eq:xy}, is such that \(x \leq c \leq y\). In particular,
  \(\P(X_0 \leq c) = p\), i.e. \(X_0 \in \cX(p;c)\), and an elementary
  computation finally shows that \(\E(X_0 \wedge c) = L_c(p)\).
\end{proof}

The following result essentially shows that Theorem \ref{thm:dlP}
implies Theorem \ref{thm:Scarf} in the unconstrained case
(i.e. without assuming that the minimizer should be bounded from
below).
\begin{prop}
  One has
  \[
    \inf_{X \in \cX} \E(X \wedge c)
    = \inf_{p \in \Pi_c} L_c(p).
  \]
\end{prop}
\begin{proof}
  The decomposition
  \[
    \cX = \bigcup_{p \in \Pi_c} \cX(p;c),
  \]
  implies (see, e.g., \cite[p.~III.11]{BBk:Ens})
  \[
    \inf_{X \in \cX} \E(X \wedge c)
    = \inf_{p \in \Pi_c}
      \inf_{X \in \cX(p;c)} \E(X \wedge c),
  \]
  so that Lemma \ref{lm:45} implies the claim.
\end{proof}

This clearly indicates that the next step should be to find the
minimum of the function \(L_c\).
\begin{lemma}
  \label{lm:mon}
  The function \(L_c\) satisfies the following properties:
  \begin{itemize}
  \item[(a)] it is decreasing on the interval
    \[
      \Bigl] 0, \frac12 + \frac12 \frac{c}{(1+c^2)^{1/2}}\Bigr];
    \]
  \item[(b)] it is increasing on the interval
    \[
      \Bigl[ \frac12 + \frac12 \frac{c}{(1+c^2)^{1/2}}, 1 \Bigr[;
    \]
  \item[(c)] admits a unique minimum point \(p_\ast\) defined by
    \[
      p_\ast := \frac12 + \frac12 \frac{c}{(1+c^2)^{1/2}}.
    \]
    Moreover, \(p_\ast\) belongs to \(\Pi_c\) and
    \[
      L_c(p_\ast) = \frac12 c - \frac12 (1+c^2)^{1/2}.
    \]
  \end{itemize}
\end{lemma}
\begin{proof}
  The function \(L_c\) is smooth and its derivative is the function
  \[
    L_c'\colon p \mapsto - \frac12 \bigl( p(1-p) \bigr)^{-1/2} (1-2p)
    - c.
  \]
  Therefore the function \(L_c\) is decreasing if and only if
  \begin{equation}
    \label{eq:segno}
    - \frac12 \frac{1-2p}{\bigl( p(1-p) \bigr)^{1/2}} \leq c,
  \end{equation}
  and increasing if and only if the sign \(\leq\) is replaced by
  \(\geq\).

  Let us first consider the case \(c \geq 0\): if \(p \leq 1/2\),
  inequality \eqref{eq:segno} is satisfied, hence \(L_c\) is
  decreasing on \(\mathopen]0,1/2\mathclose]\). If \(p>1/2\),
  inequality \eqref{eq:segno} is equivalent to
  \[
    \frac12 \frac{2p-1}{\bigl( p(1-p) \bigr)^{1/2}} \leq c,
  \]
  hence also, both sides being positive, to
  \[
    \frac{4p^2-4p+1}{p-p^2} \leq 4c^2,
  \]
  which is equivalent to
  \[
    4(1+c^2)p^2 - 4(1+c^2)p + 1 \leq 0.
  \]
  The zeros of the second-order polynomial (in \(p\)) on the left-hand
  side are
  \[
    \frac12 \pm \frac12 \frac{c}{(1+c^2)^{1/2}},
  \]
  hence the polynomial is negative on the interval
  \[
    \Bigl[ \frac12 - \frac12 \frac{c}{(1+c^2)^{1/2}}, \frac12 +
    \frac12 \frac{c}{(1+c^2)^{1/2}} \Bigr].
  \]
  We have thus shown that, if \(c \geq 0\), the function \(L_c\) is
  decreasing on the interval
  \[
    \Bigl] 0 , \frac12 + \frac12 \frac{c}{(1+c^2)^{1/2}} \Bigr]
  \]
  and increasing on the interval
  \[
    \Bigl[\frac12 + \frac12 \frac{c}{(1+c^2)^{1/2}},1 \Bigr[.
  \]
  Let us now consider the case \(c<0\): if \(p \geq 1/2\), inequality
  \eqref{eq:segno} is not satisfied, hence \(L_c\) is increasing on
  \(\mathopen[1/2,1\mathclose[\).  If \(p<1/2\), inequality
  \eqref{eq:segno} is equivalent to
  \[
    \frac{4p^2-4p+1}{p-p^2} \geq 4c^2,
  \]
  thus also to
  \[
    4(1+c^2)p^2 - 4(1+c^2)p + 1 \geq 0.
  \]
  Similarly as before, the roots of the second-order polynomial in
  \(p\) on the left-hand side are
  \[
    \frac12 \pm \frac12 \frac{\abs{c}}{(1+c^2)^{1/2}},
  \]
  hence the inequality is verified on the interval
  \[
    \Bigl[ 0, \frac12 - \frac12 \frac{\abs{c}}{(1+c^2)^{1/2}} \Bigr].
  \]
  We infer that the function \(L_c\) is decreasing on this interval
  and increasing on the interval
  \[
    \Bigl[ \frac12 - \frac12 \frac{\abs{c}}{(1+c^2)^{1/2}}, 1 \Bigr[.
  \]
  The claim then follows patching the two cases \(c<0\) and
  \(c \geq 0\) together. It only remains to show that
  \(p_\ast \in \Pi_c\): if \(c \geq 0\), this is equivalent to
  \[
    p_\ast = \frac12 + \frac12 \frac{c}{(1+c^2)^{1/2}}
    \geq \frac{c^2}{1+c^2}.
  \]
  Setting \(x := c/(1+c^2)^{1/2} \in [0,1\mathclose[\), this reduces
  to checking that \(1+x \geq 2x^2\), which is indeed the case if
  \(x \in [0,1]\).
  If \(c<0\), \(p_\ast\) belongs to \(\Pi_c\) if and only if
  \[
    p_\ast = \frac12 + \frac12 \frac{c}{(1+c^2)^{1/2}}
    = \frac12 - \frac12 \frac{\abs{c}}{(1+c^2)^{1/2}}
    \leq \frac{1}{1+c^2},
  \]
  that is, setting \(\langle c \rangle := (1+c^2)^{1/2}\) for
  convenience, if and only if
  \[
    \frac12 \frac{\abs{c}}{\langle c \rangle} \geq \frac12 -
    \frac{1}{\langle c \rangle^2},
  \]
  or, equivalently, if and only if
  \[
    \langle c \rangle \abs{c} \geq \langle c \rangle^2 - 2.
  \]
  This inequality is obviously satisfied if
  \(\langle c \rangle^2 - 2 = c^2-1 \leq 0\), i.e. if
  \(\abs{c} \leq 1\), which amounts to \(c \in [-1,0]\). If
  \(c \leq -1\), so that \(\langle c \rangle^2 - 2 \geq 0\), the
  inequality is equivalent to
  \[
    \langle c \rangle^2 c^2 \geq (c^2-1)^2,
  \]
  which simplifies to \(3c^2 \geq 1\). The inequality is then verified
  for every \(\abs{c} \geq 1/\sqrt{3} \vee 1\), i.e. for every
  \(c \leq -1\). This concludes the proof that \(p_\ast \in \Pi_c\) if
  \(c<0\), hence in general. The expression for \(L_c(p_\ast)\)
  follows by elementary algebra.
\end{proof}

We have thus solved the problem of of minimizing the function
\(X \mapsto \E(X \wedge c)\) on \(\cX\).
\begin{prop}
  Let \(p_\ast\) be defined as in Lemma \ref{lm:mon}, and
  \[
    x_\ast := - \biggl( \frac{1-p_*}{p_*} \biggr)^{1/2}, \qquad
    y_\ast := \biggl( \frac{p_*}{1-p_*} \biggr)^{1/2}.
  \]
  The two-point distributed random variable \(X_0\) with
  \[
  \P(X_0 = x_\ast) = p_\ast, \qquad \P(X_0 = y_\ast) = 1-p_\ast    
  \]
  is a minimizer of \(\inf_{X \in \cX} \E(X \wedge c)\), i.e.
  \[
    \inf_{X \in \cX} \E(X \wedge c) = L_c(p_\ast)
    = \E(X_0 \wedge c).
  \]
\end{prop}

If the parameters of the problem are such that \(X_0\), as defined in
the previous proposition, is bounded below by \(-m\), the original
minimization problem is clearly solved. We shall assume, until further
notice, that \(m\geq 0\) and \(c>-m\). This comes at no loss of
generality, as \(\cX^{-m}\) is empty if \(m<0\), and the problem
degenerates if \(c \leq -m\), in the sense that \(\E(X \wedge c) = c\)
for every \(X \geq -m\).
\begin{coroll}
  Let \(X_0\) be defined as in the previous proposition. One has
  \[
    \inf_{X \in \cX^{-m}} \E(X \wedge c) = \E(X_0 \wedge c)
  \]
  if and only if
  \[
    c \geq \frac{1-m^2}{2m}.
  \]
\end{coroll}
\begin{proof}
  By definition of \(X_0\), the lower bound \(X_0 \geq -m \) holds if
  and only if \(p_\ast \geq 1/(1+m^2)\),
  i.e. if and only if
  \[
    \frac{1}{1+m^2} \leq \frac12 + \frac12 \frac{c}{(1+c^2)^{1/2}},
  \]
  or, equivalently,
  \[
    2 \frac{1}{1+m^2} - 1 \leq \frac{c}{(1+c^2)^{1/2}},
  \]
  where the left-hand side takes values in the interval
  \(\mathopen]-1,1\mathclose]\). Elementary algebra shows that the
  inequality
  \[
    \beta \leq \frac{c}{(1+c^2)^{1/2}}, \qquad \beta \in
    \mathopen]-1,1\mathclose],
  \]
  is verified if and only if
  \[
    c \geq \frac{\beta}{(1-\beta^2)^{1/2}}.
  \]
  Replacing \(\beta\) by \(2/(1+m^2)-1\) finally implies that
  \(X_0 \geq -m\) if and only if
  \[
    c \geq \frac{1-m^2}{2m},
  \]
  from which the claim follows immediately.
\end{proof}

In view of the corollary, we only need to consider the problem under
the condition
\[
  c < \frac{1-m^2}{2m}.
\]
Note that this implies \(c < 1/m\).

Let us start observing that, for any \(X \geq -m\),
\begin{align}
  \E(X \wedge c)
  &= \E\bigl( X \indp{X \leq c} + c\indp{X>c} \bigr)\nonumber\\
  \label{eq:mcP}
  &\geq -m \P(X \leq c) + c\P(X>c)\\
  &= \E(Y \wedge c), \nonumber
\end{align}
where \(Y\) is a random variable taking values in \(\{-m,y\}\),
\(y \geq c\), with
\[
\P(Y = -m) = \P(X \leq c), \qquad \P(Y=y) = \P(X>c).
\]
In order for the random variable \(Y\) to belong to \(\cX\), it is
sufficient and necessary, in view of \eqref{eq:xy} and Lemma
\ref{lm:2p}, that
\[
  \P(Y=-m) = \P(X \leq c) = \frac{1}{1+m^2},
\]
and either \(c \leq 0\) or \(c \geq 0\) and
\[
  \P(X \leq c) \geq \frac{c^2}{1+c^2}.
\]
In other words, \(Y\) belongs to \(\cX\) and takes values in
\(\{-m,y\}\) with \(y \geq c\) if and only if \(c \leq 0\) or \(c>0\)
and
\[
  \frac{1}{1+m^2} \geq \frac{c^2}{1+c^2}.
\]
As this inequality is satisfied if and only if \(cm \leq 1\), that
holds by assumption, \(Y\) satisfies the above-mentioned conditions
if and only if \(\P(X \leq c) = 1/(1+m^2)\).
Let us then define the random variable \(Y_0 \in \cX^{-m}\) as the
(unique) random variable in \(\cX\) identified by the parameter
\[
p_m = \frac{1}{1+m^2}.
\]
We are going to show that \(Y_0\) is in fact the minimizer of the
problem.
\begin{prop}
  If \(c < (1-m^2)/(2m)\), then
  \[
    \inf_{X \in \cX^{-m}} \E(X \wedge c) = \E(Y_0 \wedge c)
    = c - (m+c)\frac{1}{1+m^2}.
  \]
\end{prop}
\begin{proof}
  Let us rewrite \eqref{eq:mcP} as
  \[
    \E(X \wedge c) \geq -c + (m+c) \P(X \leq c),
  \]
  that holds for every \(X \in \cX^{-m}\). Since \(m+c>0\) by
  assumption, the function \(p \mapsto -c + (m+c)p\) is decreasing.
  Therefore, for every \(X \in \cX^{-m}\) such that
  \(\P(X \leq c) \leq p_m\), one has
  \[
    \E(X \wedge c) \geq c - (m+c)p_m = \E(Y_0 \wedge c).
  \]
  Let \(X \in \cX^{-m}\) be such that \(p:=\P(X \leq c)>p_m\).  Then
  Lemma \ref{lm:iL} yields
  \[
    \E(X \wedge c) \geq -\bigl( p(1-p) \bigr)^{1/2} +c(1-p) = L_c(p).
  \]
  Since \(p_* < p_m\) by assumption and the function \(L_c\) is
  increasing on \(\mathopen]p_*,1\mathclose[\) by Lemma \ref{lm:mon},
  it follows that
  \[
    \E(X \wedge c) \geq L_c(p) \geq L_c(p_m) = \E(X_0 \wedge c).
  \]
  The proof is thus concluded.
\end{proof}

We have therefore proved the main result, that reads as follows.
\begin{thm}
  Let \(c,m,\sigma \in \erre\) with \(m \geq 0\), \(\sigma>0\), and
  \(c>-m\). Then
  \[
    \inf_{X \in \cX^{-m}} \E(X \wedge c) =
    \begin{cases}
      \displaystyle \E(Y_0 \wedge c) = L_c(p_m) = c - (m+c)\frac{1}{1+m^2},
      &\qquad \displaystyle \text{if } c \leq \frac{1-m^2}{2m},\\
      \displaystyle \E(X_0 \wedge c) = L_c(p_*)
      = \frac12 c - \frac12 (1+c^2)^{1/2},
      &\qquad \displaystyle \text{if } c \geq \frac{1-m^2}{2m}.
    \end{cases}
  \]
\end{thm}
Using the notation \(X(p)\) to denote a two-point distributed random
variable in \(\cX\) with parameter \(p\), one could write, more concisely,
\[
  \inf_{X \in \cX^{-m}} \E(X \wedge c)
  = \E(X(p_\ast \vee p_m) \wedge c)
  = L_c(p_\ast \vee p_m).
\]
The bound by Scarf, hence the one by Lo, i.e. Theorem \ref{thm:Scarf}
and its corollary, follow immediately by the previous theorem and
Lemma \ref{lm:redu}.

\bibliographystyle{amsplain}
\bibliography{ref,finanza}

\end{document}